\documentclass{amsart}
\usepackage{amssymb}
\usepackage{amscd}
\usepackage[all,cmtip]{xy}
\usepackage{hyperref}
\usepackage{graphicx}

\newcommand{\comment}[1]{}

\newcommand{\Ps}{\mathbf{P}} 
\newcommand{\Q}{\mathbf{Q}}
\newcommand{\Z}{\mathbf{Z}}

\newcommand{\F}{\mathbf{F}}

\newcommand{\m}{\mathfrak{m}}
\newcommand{\pa}{\mathfrak{p}}
\newcommand{\qa}{\mathfrak{q}}

\newcommand{\sep}{\mathrm{sep}}

\DeclareMathOperator{\Spec}{Spec}

\newcommand{\limplies}{\Longleftarrow}

\theoremstyle{plain}
\newtheorem{theorem}{Theorem}[section]

\newtheorem{corollary}[theorem]{Corollary}

\newtheorem{lemma}[theorem]{Lemma}

\newtheorem{proposition}[theorem]{Proposition}

\theoremstyle{definition}
\newtheorem{definition}[theorem]{Definition}
\newtheorem{remark}[theorem]{Remark}

\makeindex

\begin{document}

\title[Futile algebras]{Algebras with only finitely many subalgebras}
\author{Michiel Kosters}
\address{Mathematisch Instituut
P.O. Box 9512
2300 RA Leiden
The Netherlands}
\email{mkosters@math.leidenuniv.nl}
\urladdr{www.math.leidenuniv.nl/~mkosters}
\date{\today}
\subjclass[2010]{}
\thanks{I would like to thank Hendrik Lenstra for his help.}

\begin{abstract}
Let $R$ be a commutative ring. A not necessarily commutative $R$-algebra $A$ is called futile if it has only finitely many
$R$-subalgebras. In this article we relate the notion of futility to familiar properties of
rings and modules. We do this by first reducing to the case where $A$ is commutative. Then we refine the description of
commutative futile algebras from Dobbs, Picavet and Picavet-L'hermite.
\end{abstract}

\maketitle
\tableofcontents

\section{Introduction}

In the whole article, let $R$ be a commutative ring. If $R$ is a domain, we denote by $Q(R)$ its quotient field. For an $R$-module $M$, where $R$ is a domain, we denote $M_{R\textrm{-tor}}=\{m \in M: \exists r \in R \setminus \{0\}, rm=0\}$. 

An $R$-algebra is by definition a not necessarily commutative
ring $A$ together with a ring homomorphism $\varphi: R \to A$ such that the image of $\varphi$ is contained in the center $Z(A)$ of $A$.  By abuse of
notation we will often write $R$ instead of $\varphi(R)$. For example $A/R$ means $A/\varphi(R)$. 
We will reserve the word ring for a commutative ring.

\begin{definition}
 An $R$-algebra $A$ is called $R$-futile if it has only finitely many $R$-subalgebras. We sometimes just say that $A$ is futile if it is clear to
which $R$ we refer.
\end{definition}

Given $R$, we want to describe the futile $R$-algebras in terms of familiar intrinsic properties of rings and modules.
We first
reduce to the case where our algebras are commutative. The proof of the following two theorems can be found in
Section \ref{307}. For an $R$-algebra
$A$ we define the commutator ideal to be the two-sided ideal $[A,A] \subset A$ generated by $[a,b]=ab-ba$ for $a,b \in A$. Notice that $A/[A,A]$ is
commutative.

\begin{theorem} \label{130}
 An $R$-algebra $A$ is $R$-futile if and only if $A/[A,A]$ is a futile $R$-algebra and $[A,A]$ is finite.
\end{theorem}

An $R$-algebra $A$ is called monogenic if there exists $a
\in A$ with $A=R[a]$, where $R[a]$ is the smallest $R$-subalgebra of $A$ containing $a$.  The following theorem gives conditions when all futile
algebras are commutative. 

\begin{theorem} \label{24}
The following statements are equivalent:
\begin{enumerate}
\item all futile $R$-algebras are monogenic over $R$;
\item all commutative futile $R$-algebras are monogenic over $R$;
\item all futile $R$-algebras are commutative;
\item for all $\m \in \mathrm{Spec}(R)$ the ring $R/\m$ is infinite.
\end{enumerate}
\end{theorem}

The case where our $R$-algebra $A$ is assumed to be commutative, has been studied intensively before by various authors and has resulted in
\cite{DO5}. In their work, one says that a commutative $R$-algebra $A$ satisfies FIP if $A$ is $R$-futile. In Section \ref{66} we will discuss their
work. In this theory, one reduces in some cases to the case where $R$ is an infinite local
artinian ring. The case that $R$ is local artinian, was handled in \cite{DO5}, but we provide a different treatment of this case. The following three
theorems summarize our results. Proofs can be found in Section \ref{1234}. 

The first theorem discusses when an extension of fields is futile.

\begin{theorem} \label{401}
 Let $L/K$ be an extension of fields. Let $p$ be the characteristic of $K$ if the characteristic is nonzero and $1$ otherwise. Then the following are
equivalent:
\begin{enumerate}
 \item $L$ is a futile $K$-algebra;
 \item $[L:K] < \infty$ and $[L:L^p K] \in \{1,p\}$;
 \item $L=K[\alpha]$ for some $\alpha \in L$ (the field extension is primitive).
\end{enumerate}
\end{theorem}

The second theorem describes the futile $R$-algebras when $R$ is an infinite field. For a commutative ring $S$ we put $\sqrt{0}_S$ for the set of
nilpotent elements in $S$.

\begin{theorem} \label{3}
Let $R$ be an infinite field. Then the following properties are equivalent for an $R$-algebra $A$:
\begin{enumerate}
 \item $A$ is a futile $R$-algebra;
 \item  $A \cong_R A' \times \prod_{i \in I} A_i$ where $I$ is a finite size and the $A_i$ are finite primitive field extensions of $R$ and $A'$ is a
commutative $R$-algebra which
satisfies $\mathrm{dim}_R(A') \leq 3$ and if $\mathrm{dim}_R(A')=3$, then $\sqrt{0_{A'}}^2 \neq 0$;
 \item $A \cong_R R[x]/(f)$ where $f \in R[x]$ splits into irreducible factors
$f= \prod_{i=1}^m f_i^{n_i}$ where all the $f_i$ are monic, pairwise coprime, $n_i=1$ for all but possibly one $i$ in which case $\mathrm{deg}(f_i)=1$
and $n_i \leq 3$.
\end{enumerate}  
\end{theorem}

The third theorem describes futile $R$-algebras where $R$ is an infinite local artinian ring. It makes use of the previous theorem. An $R$-module
$M$ is called uniserial if the $R$-submodules of $M$ are ordered linearly by inclusion.

\begin{theorem} \label{913}
Let $(R,\m)$ be a local artinian ring such that $k=R/\m$ is infinite and let $A$ be an $R$-algebra. Put $T=R+\sqrt{0_A}$ with maximal ideal
$\mathfrak{n}=\sqrt{0_A}$ and put $r_T= \mathrm{dim}_{R/\m}(\sqrt{0_{T/\m T}})$. Then the following properties are
equivalent.
\begin{enumerate}
 \item $A$ is a futile $R$-algebra;
 \item $A$ is commutative, $A/\m A$ and $T/\m T$ are futile $k$-algebras, $\m (A/R)$ is a uniserial $R$-module and
if $r_T=2$, then one has $\mathfrak{n}^4+\mathfrak{n}^2\m+\m=\m T$ in $T$.
\end{enumerate}
\end{theorem}

Finally, we will prove a theorem which summarizes the results for $R=\Z$ (Theorem \ref{21}).

\begin{theorem}
A $\Z$-algebra $A$ is $\Z$-futile if and only if one of the following holds:
\begin{itemize}
\item $A$ is finite;
\item  $A_{\Z\textrm{-}\mathrm{tor}}$ is
finite and
$A/A_{\Z\textrm{-}\mathrm{tor}} \cong \Z[1/n] \subset \Q$ for some $n \in \Z \setminus \{0\}$.
\end{itemize}
\end{theorem}

\section{General statements} \label{307}

In this section we will prove certain statements which hold for any commutative ring $R$. Throughout this section, we let $A$ be a not necessarily
commutative $R$-algebra with morphism $\varphi: R \to A$.

\begin{theorem} \label{120}
 Let $R_1,\ldots, R_n$ be rings ($n \in \Z_{\geq 1}$). Put $R= \prod_{i=1}^n R_i$. Then we have an equivalence of categories
\begin{eqnarray*}
\varphi: \mathrm{Alg}_{R_1} \times \ldots \times \mathrm{Alg}_{R_n} &\to& \mathrm{Alg}_{R} \\
(A_1,\ldots,A_n) &\mapsto& A_1 \times \ldots \times A_n \\
(\varphi_1, \ldots,\varphi_n) &\mapsto& \varphi_1 \times \ldots \times \varphi_n.
\end{eqnarray*}
\end{theorem}
\begin{proof}
 The inverse is given by $A \mapsto (A \otimes_R R_i)_{i=1}^n$. The rest is easy.
\end{proof}

The above theorem allows us to reduce to the case where $R$ is a connected ring if $R$ has finitely many idempotents.

\begin{lemma}
 Assume that the index $(A:R)=\#A/R$ is finite. Then $A$ is a futile $R$-algebra.
\end{lemma}
\begin{proof}
Consider the injective map from the set of $R$-subalgebras of $A$ to the power set of $A/R$ given by $B \mapsto \mathrm{Im}(B \to
A/R)$.
\end{proof}

We have the following easy observations.

\begin{lemma} \label{0}
The following statements hold:
\begin{enumerate}
 \item 
 Assume that $A$ is $R$-futile. Then one has:
 \begin{enumerate}
 \item Any $R$-subalgebra of $A$ is $R$-futile.
 \item Let $I \subseteq A$ be a two sided ideal of $A$. Then $A/I$ is a futile $R$-algebra.
 \item Let $I \subseteq R$ be an ideal. Then $A/IA$ is a futile $R/I$-algebra.
\end{enumerate}
 \item
 Let $I \subseteq R$ be a common ideal of $R$ and $A$. Then $A$ is $R$-futile if and only if $A/I$ is $R/I$-futile. 
 \item
Let $S$ be any multiplicative subset of $R$ and let $\varphi: A \to S^{-1}A$. The map from $S^{-1}R$-subalgebras of $S^{-1}A$ to $R$-subalgebras of
$A$, given by $B \mapsto \varphi^{-1}(B)$, is injective and respects inclusions. If $A$ is $R$-futile, then $S^{-1}A$ is $S^{-1}R$-futile. 
\end{enumerate}
\end{lemma}
\begin{proof}
i. Statement a is obvious. For statement b, let $\psi: A \to A/I$. Then the inverse of an $R$-algebra of $A/I$ is an $R$-algebra of $A$ containing
$I$. For statement c, notice that by b $A/IA$ is a futile $R$-algebra. Notice that an $R$-subalgebra in this case is
the same as an $R/I$ subalgebra.

ii. Obvious.

iii. This easily follows from $S^{-1}\varphi^{-1}(B)=B$.

\end{proof}

\begin{lemma} \label{7}
Let $n \in \Z_{\geq 1}$. Let $G$ be a group and for $1 \leq i \leq n$ let $g_i \in G$ and $H_i \subseteq G$ be subgroups. Suppose that
$G=\bigcup_{i=1}^n g_iH_i$. Then one has
$G=\bigcup_{i:\ (G:H_i)<\infty}g_iH_i$.  
\end{lemma}
\begin{proof}
 See \cite{RO}, Lemma 4.17.
\end{proof}

The following lemma is very useful.

\begin{lemma} \label{1}
 Assume that $A$ is $R$-futile. Then there exists $n \in \Z_{\geq 0}$ and $\alpha_i \in A$ ($i=1,\ldots,n$) such that $A= \bigcup_{i=1}^n
R[\alpha_i]$ and $(A:R[\alpha_i])<\infty$. 
\end{lemma}
\begin{proof}
 Notice that $A=\bigcup_{a \in A} R[a]$, and as $A$ is a futile $R$-algebra only finitely many of the $R$-algebras $R[a]$ are needed. Use Lemma
\ref{7} to
finish the proof. 
\end{proof}

\begin{lemma} \label{17}
Let $I \subseteq A$ be a two sided ideal of finite order. Then $A$ is a futile $R$-algebra if and only if $A/I$
is a futile $R$-algebra. 
\end{lemma}
\begin{proof}
 $\implies$: See Lemma \ref{0}ib. 

$\limplies$: Let $C$ be an
$R$-subalgebra of $A$. Note that $(C+I)/I$ is an $R$-subalgebra of $A/I$. As $A/I$ is $R$-futile, it is enough to show that there are only finitely
many $R$-subalgebras of $A$ mapping to such a given $(C+I)/I$.
Suppose that for another such algebra $C'$ we have $C+I=C'+I$. As $A/I$ is $R$-futile, it follows that $(C+I)/I$ is finitely generated, say by $c_i+I$
($i=1,\ldots,n$, $c_i \in C$). There are $d_i \in C'$ such that $c_i \in d_i+I$. We have
\begin{eqnarray*}
R[d_i: i=1,\ldots,n] \subseteq C' \subseteq C+I=R[d_i: i=1,\ldots,n]+I.
\end{eqnarray*}
As $I$ is finite, given the $d_i$, this gives only finitely many options for $C'$. As $I$ is finite, there are finitely many options for the
$d_i$. Hence the result follows.
\end{proof}

\begin{proof}[Proof of Theorem \ref{24}]
i $\implies$ ii: Trivial.

i $\implies$ iii: Monogenic rings over commutative rings are commutative.

iii $\implies$ iv: Suppose that $\m \in \mathrm{Spec}(R)$ is such that $R/\m$ is finite. Then we have a finite non-commutative $R$-algebra
$\mathrm{Mat}_2(R/\m)$ which is a futile $R$-algebra.

iv $\implies$ i: Let $A$ be a futile $R$-algebra. Then write $A= \bigcup_{i=1}^n R[a_i]$ for $a_i \in A$
where
$(A:R[a_i])<\infty$ (Lemma \ref{1}) and $n \in \Z_{\geq 1}$. But then $A/R[a_1]$ is a finite $R$-module. The only finite $R$-module under our
assumptions is $0$. Hence we find $A/R[a_1]=0$ and $A=R[a_1]$. 

ii $\implies$ iv: Suppose that for some $\m \in \mathrm{Spec}(R)$ the ring $k=R/\m$ is finite of size $n$. Consider $k^{n+1}$, which is a finite
ring and hence a futile $R$-algebra. We claim that it is not monogenic. Indeed, otherwise there if an $f \in k[x]$ such that $k[x]/(f(x))
\cong k^{n+1}$, but $f$ cannot have
enough different linear factors to make this possible. 
\end{proof}

\begin{remark}
 In \cite{DO4} Proposition 3.1 it has been shown that any commutative $R$-algebra which is a futile $R$-algebra is monogenic if $R$ contains
an infinite set $S$
of units such that $u-v \in R^* \cup \{0\}$ for all $u,v \in S$. One easily sees that this condition implies that for all $\m \in
\mathrm{Spec}(R)$ quotient $R/\m$ is infinite. Hence this condition implies the condition in Theorem \ref{24}. The converse is not true. For example,
one can consider
the ring \[R=\F_2[X_n,Z_n: n \in \Z_{\geq 1}][\frac{U_n}{V_n}: n \in \Z_{\geq 1}].\] where
$U_n=1+Z_n^{2^n}-Z_n$ and $V_n=X_n^{2^n}-X_n$. One has $R^*=\{1\}$ in this case, but for any $\m \in \mathrm{Spec}(R)$ the quotient $R/\m$ is
infinite.
\end{remark}

\begin{lemma} \label{100}
 Let $A$ be a ring and assume that $(A:Z(A))<\infty$. Then the commutator ideal $[A,A]$ is finite.
\end{lemma}
\begin{proof}
 We have a natural map
\begin{eqnarray*}
[\ ,\ ]: A/Z(A) \otimes_{Z(A)} A/Z(A) &\to& A \\
 \overline{a} \otimes \overline{b} &\mapsto& ab-ba. 
\end{eqnarray*}
As $A/Z(A)$ is finite, so is the left hand side and hence the image of this map. Call this image $B$. Consider the exact sequence
$0 \to Z(A) \to
A \to A/Z(A) \to 0$. Now tensor this sequence with $B$ over $Z(A)$ to obtain the exact sequence $B \to A \otimes_{Z(A)} B \to
A/Z(A) \otimes_{Z(A)}
B \to 0$. Both $B$ and $A/Z(A) \otimes_{Z(A)} B$ are finite, and hence so is $A \otimes_{Z(A)} B$. Notice that the map $A \otimes_{Z(A)} B \to BA$ is
surjective. Note that $BA=AB=[A,A]$ due to the identity $a[x,y]=[x,y]a+[a,[x,y]]$ and hence the ideal $[A,A]$ is finite. 
\end{proof}

\begin{lemma} \label{9} 
Suppose that $A$ is $R$-futile. Then $Z(A)$ is of finite index in $A$ and the commutator ideal, $[A,A]$, is finite.
\end{lemma}
\begin{proof}
 Write $A=\bigcup_{i=1}^n R[\alpha_i]$ where $\alpha_i \in A$ and $(A:R[\alpha_i])<\infty$ (Lemma \ref{1}). Notice that $\bigcap_{i=1}^n R[a_i]
\subseteq Z(A)$ and that this is of finite index in $A$. Now apply Lemma \ref{100}.
\end{proof}

\begin{proof}[Proof Theorem \ref{130}]
Note that $[A,A]$ is finite by Lemma \ref{9}. Apply Lemma \ref{17}. 
\end{proof}

\begin{lemma}
Then the following statements are equivalent:
\begin{enumerate}
\item
for every ring morphism $R \to R'$ the $R'$-algebra $A \otimes_R R'$ is $R'$-futile;
\item
$A$ is a quotient of $R$.
\end{enumerate}
\end{lemma}
\begin{proof}
i $\implies$ ii: Take $R'=R[x]$. Hence we are given that $A[x]=A \otimes_R R[x]$ is a futile $R[x]$-algebra. Suppose that we have an $a \in A \setminus R$. For $i \in \Z_{\geq 1}$ consider the rings $B=R[x]+ax^i A[x]$. This gives us
infinitely many $R[x]$-subalgebras, which contradicts the futility. 

ii $\implies$ i: If $A=R/I$, then $A \otimes_R R' = R'/IR'$, which is obviously $R'$-futile.
\end{proof}

\section{Commutative case} \label{66}

In this section we summarize the theory of commutative futile $R$-algebras as developed in \cite{DO5}. We have adapted some of the statements 
in order to make them easier to read. 
In \cite{DO5}, and some other articles, one says shat a commutative $R$-algebra $A$ with $A \supseteq R$ satisfies FIP if it is $R$-futile.

In this section we let $S$ be a commutative futile $R$-algebra with $R \subseteq S$. The latter is not really a
restriction, because we can replace $R$ by its image in $S$. 

We put $\tilde{R}$ for the integral closure of $R$ in $S$.
\begin{theorem}
The algebra $S$ is $R$-futile if and only if $\tilde{R}$ is $R$-futile and
$S$ is $\tilde{R}$-futile. 
\end{theorem}
\begin{proof}
 See \cite{DO5}, Theorem 3.13. 
\end{proof}

Hence our problem reduces to two cases: the case where $R \subseteq S$ is integral and the case where $R=\tilde{R}$. For an $R$-module $M$ we
put $\mathrm{MSupp}(M)=\{\m \in \mathrm{MaxSpec}(R): M_{\m} \neq 0\}$. For an inclusion of rings $A \subseteq A'$ we put $(A:A')=\{ x \in A':
xA' \subseteq A\}$, which is the largest common ideal of both $A$ and $A'$. 

\begin{theorem} \label{901}
 Suppose that $R=\tilde{R} \subsetneq S$. Then $S$ is $R$-futile if and only if the following properties hold:
\begin{enumerate}
 \item $\mathrm{MSupp}_{R}(S/R)$ is a finite set;
 \item for every $\m \in \mathrm{MSupp}_{R}(S/R)$, the ideal $\pa=(R:S)_{\m} \subseteq R_{\m} $ is prime, $S_{\m}=(R_{\m})_{\pa}$ and $R_{\m}/\pa$ is
a valuation ring of finite Krull-dimension. 
\end{enumerate}
\end{theorem}
\begin{proof}
Theorem 6.16 and the references in its proof from \cite{DO5} state the following. The algebra $S$ is $R$-futile iff $\mathrm{MSupp}_{R}(S/R)$ is a
finite set and for every $\m \in \mathrm{MSupp}_{R}(S/R)$ there exists $\pa \in \Spec(R_{\m})$ such that $S_{\m}=(R_{\m})_{\pa}$, $\pa=S_{\m} \pa$ and
$R_{\m}/\pa$ is a valuation ring of finite Krull-dimension.

We show that our statement is equivalent to this theorem. First assume our statement (i and ii). Given $\m \in \mathrm{MSupp}_{R}(S/R)$, consider
$\pa=(R:S)_{\m}$. We just need to show that $\pa=S_{\m} \pa$. But $(R:S)_{\m}=(R_{\m}:S_{\m})$ and hence $\pa=S_{\m} \pa$.

Conversely, given $\m \in \mathrm{MSupp}_{R}(S/R)$, suppose that $\pa \in \Spec(R_{\m})$ satisfies the assumptions as in Theorem 6.16 from \cite{DO5}.
As $\pa=S_{\m} \pa$, $\pa$ is an ideal in $S$ and we have $\pa \subseteq (R_{\m}:S_{\m})=(R:S)_{\m}$. As
$Q(R_{\m}/\pa)=(R_{\m})_{\pa}/\pa=S_{\m}/\pa S_{\m}$, it follows that $\pa$ is a maximal ideal of $S$. As $(R:S)_{\m} \subsetneq R_{\m}$ by
assumption, the result follows.
\end{proof}

This settles the first case. For the integral part, we will reduce to the case where $R$ is local artinian. We first need
two lemmas.

\begin{lemma} \label{113}
Let $f: A \to B$ be a morphism of rings which makes $B$ into a finitely generated $A$-module. Let $M$ be a $B$-module. Then
$\mathrm{length}_B(M)<\infty$ implies
$\mathrm{length}_A(M)<\infty$. 
\end{lemma}
\begin{proof}
Let $\m \subset B$ be a maximal ideal. Then we need to show that $\mathrm{length}_A(B/\m)=\mathrm{length}_{A/f^{-1}(\m)}(B/\m)$ is finite. As $B/\m$
is a finite field extension of $A/f^{-1}(\m)$ (Corollary 5.8 from \cite{AT}), the result follows. 
\end{proof}

\begin{lemma} \label{114}
 Let $R$ be a ring and let $M$ be an $R$-module. Then $\mathrm{length}_R(M)<\infty$ implies $R/\mathrm{Ann}_R(M)$ is artinian. The converse is true
if $M$ is finitely generated as $R$-module.
\end{lemma}
\begin{proof}
We will prove the first statement. It follows that $M$ is finitely generated and we have an embedding $R/\mathrm{Ann}_R(M) \to M^n$ for some $n$.
Hence $R/\mathrm{Ann}_R(M)$ has finite length and the result follows from \cite{EI} Theorem 2.14. 

For the converse, we have a surjective map $(R/\mathrm{Ann}_R(M))^n \to M$ for some $n \in \Z_{\geq 0}$ where the domain is of finite length. 
\end{proof}

\begin{theorem} \label{111}
Let $R \subsetneq S$ be integral. Then $S$ is $R$-futile if and only if $R/(S:R)$ is artinian and $S/(S:R)$ is $R/(S:R)$-futile.
\end{theorem}
\begin{proof}
 See Theorem 4.2 from \cite{DO5}. We will give a similar proof.

$\implies$: The last part follows from Lemma \ref{0}ic. By Lemma \ref{114} it is enough to show $\mathrm{length}_R(S/R)<\infty$. Using Lemma
\ref{113} we may assume
that there
are no subrings strictly between $R$ and $S$. Furthermore, we may assume that $(S:R)=0$. 
Let $\m$ be a maximal ideal such that $A_{\m}\to B_{\m}$ is not an isomorphism (\cite{AT} Proposition 3.9). Note that there are still no non-trivial
subrings between
$R_{\m}$ and $S_{\m}$ (Lemma \ref{0}iii). Suppose that $\m S
\not \subseteq R$, then $S_{\m}=R_{\m}+\m S_{\m}$. Hence by Nakayama's Lemma (\cite{AT} Proposition 2.6) we conclude $R_{\m}=S_{\m}$, a
contradiction. Hence $\m \subseteq (S:R)=0$ and $R$ is a field. Since $S$ is finitely generated and integral over a field $R$,
$\mathrm{length}_R(S/R)<\infty$ as required. 

$\limplies$: See Lemma \ref{0}ii. 

\end{proof}

This reduces the problem to the case where $R$ is artinian. As an artinian ring is a product of local artinian rings, and the futility property
behaves well with respect to products on the base (Theorem \ref{120}), we may assume that $(R,\m)$ is local artinian. There are again two cases: the
residue field is finite or infinite. We first treat the case where the residue field is finite.

\begin{theorem}
 Let $R \subseteq S$ be integral with $(R,\m)$ local artinian with $R/\m$ finite. Then $S$ is a futile $R$-algebra if and only if $S$ has finite
size.
\end{theorem}
\begin{proof}
See Theorem \ref{31}, since $R$ is of finite size. 
\end{proof}

We consider the case where $R$ is local artinian with infinite residue field. From Theorem \ref{111} we see that we may assume that $(R:S)=0$. The
following is a more polished version of Theorem 5.18 from \cite{DO5}. 

\begin{theorem} \label{112}
 Let $R \subseteq S$ be integral with $(R,\m)$ local artinian with infinite residue field and $(R:S)=0$. 
Put $T=R+\sqrt{0}_S$ and $R'=R+T\m$. 
Then $S$ is a futile $R$-algebra if and only if the following
properties hold:
\begin{enumerate}
 \item $S$ is finitely generated as an $R$-algebra;
 \item there exists $\gamma \in S$ such that $S=T[\gamma]$;
 \item $\m T/\m$ is a uniserial $R$-module;
 \item there exists $\alpha \in T$ such that $T=R'[\alpha]$ and $\alpha^3 \in T\m$, and, with $T'=R'[\alpha^2]$ and $T''=R+T'\m$, there exists
$\beta \in T$ such that $T'=T''[\beta]$ and $\beta^3 \in T'\m$. 
\end{enumerate}
\end{theorem}
\begin{proof}
This follows Theorem 5.18 from \cite{DO5} keeping in mind that a futile $R$-algebra coming from an integral extension is finite as $R$-module, and
under this assumption, FCP follows directly. Also use Lemma \ref{90} and notice that the length condition is automatically satisfied. 
\end{proof}

Theorem \ref{401}, Theorem \ref{3} and Theorem \ref{913} give an alternative to Theorem \ref{112}.

\section{Artinian rings} \label{1234}

\subsection{Finite rings}

\begin{theorem} \label{31} \label{2}
 Let $R$ be an artinian ring and let $A$ be a futile $R$-algebra. Then $A$ is finite as $R$-module. If $R$ is of finite size, then so is $A$. 
\end{theorem}
\begin{proof}
We can reduce to the case where $R$ is local by using Theorem \ref{120} and Theorem 8.7 from \cite{AT}. Let $a \in A$ and consider the
subalgebras $R[a^i]$ for $i \in \Z_{\geq 2}$. As $A$ is a futile $R$-algebra, there are $m$ and $n$ coprime such that
$R[a^m]=R[a^n]$. Hence we see that $a^m=\sum_{i=1}^s r_ia^{in}$. This shows that there is a polynomial $f \in R[x]$ with some unit coefficient which
satisfies $f(a)=0$. Write $f=g-h$ where the coefficients of $g$ are units and the coefficients of $h$ are nilpotent. Take a $t \in \Z_{\geq 0}$
such that $h^t=0$. Then, as $g(a)=h(a)$ we have $g(a)^t=h(a)^t=0$. The highest
coefficient of $g$ is still a unit, and hence it follows that $R[a]$ is a finite $R$-module. From the futility it follows that $R$ is a finite union
of $R$-modules of finite length, and hence that $A$ is a finite $R$-module.

The last statement follows directly.
\end{proof}

\subsection{Extensions of fields} \label{800}

Let $L/K$ be an extension of fields and let $p$ be the characteristic of $K$ if this is nonzero, and $1$ otherwise.
Then we put $L_i=\{x \in L: \exists j: x^{p^j} \in K\}$, the maximal purely inseparable field extension of $K$
in $L$. Put $L_s=\{x \in L: x \textrm{ separable over } K\}$. Notice that $L_i \cap L_s=K$. 

\begin{definition}
 A field extension $L/K$ is called separably disjoint if $L=L_sL_i$. Equivalently, $L/K$ is separably disjoint if $L/L_i$ is separable.
\end{definition}

One can easily show that a normal extension is separably disjoint by using Galois theory (Proposition 6.11 from \cite{LA}). 

Notice that a field extension $L/K$ has a unique maximal separably disjoint subextension, namely $L_sL_i$.

\begin{lemma} \label{1111}
 Let $L/K$ be an algebraic extension of fields. Then the map
\begin{eqnarray*}
\varphi: \{E: K \subseteq E \subseteq L\} &\to& \{E': K \subseteq E' \subseteq L_{K,\sep}\} \times \{E'': L_{K,\sep} \subseteq E'' \subseteq
L\} \\
E &\mapsto& (E \cap L_{K,\sep}, E L_{K,\sep})
\end{eqnarray*}
is injective. The image consists of pairs $(E_1,E_2)$ with $E_1 \subseteq E_2$ and $E_2$ separably disjoint over $E_1$. 
\end{lemma}
\begin{proof}
 We will construct $E$ from $(E \cap L_{K,\sep}, E L_{K,\sep})$. Let $E'=\{x \in E L_{K,\sep}: \exists j: x^{p^j} \in E \cap
L_{K,\sep}\}$.
We claim that
$E=E'$. One easily obtains $E \subseteq E'$. Let $x \in EL_{K,\sep}$ such that $x^{p^j}
\in E \cap L_{K,\sep}$. As $E
L_{K,\sep}/E$ is separable, it follows that $x \in E$. 

For $(E_1,E_2)$ in the image, one easily obtains that $E_2/E_1$ is separably disjoint. Indeed, $E/E \cap L_{K,\sep}$ is purely inseparable,
$L_{K,\sep}/E \cap L_{K,\sep}$
is separable and their compositum is $EL_{K,\sep}$. On the other hand, consider a pair $(E_1,E_2)$ with $E_1 \subseteq E_2$ and $E_2/E_1$ separably
disjoint. Set $N=\{x \in E_2: \exists j: x^{p^j} \in E_1\}$. One then easily deduces $\varphi(N)=(E_1,E_2)$.
\end{proof}

Assume that $[L:K]<\infty$. Let $j \in \Z_{\geq 0}$. We have $[L^{p^j}K:L^{p^{j+1}}K]=[L^{p^{j+1}}K^p:L^{p^{j+2}}K^p] \geq
[L^{p^{j+1}}K:L^{p^{j+2}}K]$. Let $j$ be the first $j$ such that $[L^{p^j}K:L^{p^{j+1}}K]=1$. Then obviously $L^{p^j}K$ is separable over $K$ and it
is the separable closure of $L$ in $K$. 

\begin{proof}[Proof of Theorem \ref{401}]
 i $\implies$ iii: If $K$ is finite, the statement follows from Theorem \ref{31} and the fact that finite extensions of finite fields are primitive. If
$K$ is infinite, use Theorem \ref{24}.

 iii $\implies$ ii: Note that $L/K$ is automatically finite. Also $L^pK=K(\alpha^p)$ and one easily sees $[K(\alpha):K(\alpha^p)]\in \{1,p\}$. 

ii $\implies$ i: Notice that $K$-subalgebras of $L$ are automatically fields. Using Lemma \ref{1111} it is enough to show that $L_s/K$ and $L/L_s$ both have finitely many subfields.
Notice that $L_s/K$ has finitely many subextensions by Galois theory. Consider the purely inseparable extension $L/L_s$. As $[L:L^pK] \in
\{1,p\}$, one easily sees that all the subfields of $L/L_s$ are given by $L^{p^0}K \supsetneq L^{p^1}K \supsetneq \ldots \supsetneq L^{p^i}K=L_s$
where $[L:L_s]=p^i$.
\end{proof}

\subsection{Infinite fields}

We will now study futile $R$-algebras where $R$ is an infinite field. Most results of this section were known before (see for example \cite{DO1}), but
the proofs are different.

\begin{lemma} \label{4}
Let $R$ be an infinite field and let $f \in R[x]$. 
\begin{enumerate}
 \item Assume that $\mathrm{deg}(f)=1$. Let $r \in \Z_{\geq 1}$. Then $A=R[x]/(f^r)$ is a futile $R$-algebra if and only if $r \leq 3$. 
 \item Assume that $n=\mathrm{deg}(f)>1$. Let $r \in \Z_{\geq 2}$. Then $A=R[x]/(f^r)$ is not a futile $R$-algebra.
 \item Assume that $f$ is irreducible in $R[x]$. Then $R[x]/(f)$ is a futile $R$-algebra.  
\end{enumerate}
Furthermore, the $R$-subalgebras of $R[x]/(x^i)$ where $i \in \{0,1,2,3\}$ are $R[x^j] \subseteq R[x]/(x^i)$ for $j=1,\ldots,i$. 
\end{lemma}
\begin{proof}
 ii. By Lemma \ref{0}ib we may assume that $r=2$. Consider the following
map:
\begin{eqnarray*}
 \Ps^{n-1}(R) &\to& \{R\textrm{-subalgebras of }A \} \\
(a_0:\ldots:a_{n-1}) &\mapsto& R \oplus (f \cdot \sum_{i=0}^{n-1} a_i x^i).
\end{eqnarray*}
Notice that this map is injective and that, as $n \geq 2$, the set $\Ps^{n-1}(R)$ is infinite.

 i. $\implies$: This follows from and Lemma \ref{0}ib and ii, where we take $f^2$ instead of $f$.

$\limplies$: We show that the statement is true if $r=3$, the rest follows from Lemma \ref{0}ib. After a translation we may assume that $f=x$ and that
$A=R[x]/(x^3)$. We claim that the only $R$-subalgebras are $R$, $A$ and the ring generated by $R$ and $x^2$. Indeed, consider the ring
generated by $g=a_0+a_1x+a_2x^2$ over $R$. We may assume that $a_0=0$. If $a_1 \neq 0$, we may assume that $a_1=1$ and we have $x^2=g-a_2g^2$. Hence the ring generated by $g$
is just $A$. If $a_1=0$ and $a_2 \neq 0$, then the ring is generated by $x^2$. The statement follows. Furthermore, the last statement also follows easily.

iii. This follows from Theorem \ref{401}.

\end{proof}

The following lemma allows us to work with products of algebras.

\begin{lemma}[Goursat] \label{5}
 Let $A, B$ be $R$-algebras. Then there is a bijection between the quintuples $(C,D,I,J,\varphi)$ with the following properties
\begin{itemize}
\item $C$ is an $R$-subalgebra of $A$;
\item $D$ is an $R$-subalgebra of $B$;
\item $I \subseteq C$ is a two-sided ideal;
\item $J \subseteq D$ is a two-sided ideal;
\item an $R$-algebra isomorphism $\varphi: C/I \overset{\sim}{\to} D/J$;
\end{itemize}
and the set of $R$-subalgebras of $A \times B$ given by $(C,D,I,J,\varphi) \mapsto \{(a,b) \in C \times D: \varphi(\overline{a})=\overline{b}\}$. 
\end{lemma}
\begin{proof}
 The proof is essentially the same as the proof of Goursat's Lemma for groups. 
\end{proof}

\begin{lemma} \label{11}
 Let $A,B$ be futile $R$-algebras. Suppose that for any quotient $C$ of an $R$-subalgebra of $A$
we have that $\#
\mathrm{Aut}_R(C)<\infty$ and that subalgebras of $A$ respectively $B$ have only finitely many ideals. Then $A \times B$ is a futile
$R$-algebra.
\end{lemma}
\begin{proof}
 This follows from Lemma \ref{5}.
\end{proof}

\begin{lemma} \label{12}
 Let $R$ be a field and let $F=\prod_{i=1}^n F_i$ ($n \in \Z_{\geq 0}$) an $R$-algebra where the $F_i$ are fields and $[F_i:R]<\infty$. Then we have:
\begin{enumerate}
 \item any $R$-subalgebra of $F$ is a finite product of fields which are finite over $R$;
 \item $F$ has only finitely many ideals and a quotient by such an ideal is isomorphic to a product
of fields which are finite over $R$;
 \item $\# \mathrm{Aut}_R(F)<\infty$.
\end{enumerate}
\end{lemma}
\begin{proof}
 i. Let $A$ be a subalgebra. Then $A$ is artinian and hence isomorphic to a product of local artinian rings. Notice that a local reduced artinian
ring is a field.

 ii. This follows easily because we know the ideals of $F$. 

 iii. This follows easily by looking at stalks and the fact that $\#\mathrm{Aut}_R(F_i)<\infty$.

\end{proof}

\begin{proof}[Proof of Theorem \ref{3}]
i $\implies$ iii: Suppose that $A$ is a futile $R$-algebra. By Theorem \ref{24} we know that $A=R[a]$ for some $a \in A$. Note
that $R[x]$ is a principal ideal domain, so there is a non-zero polynomial $f$ such that $R[a] \cong R[x]/(f)$. Write $f= \prod_{i=1}^m f_i^{n_i}$
where all the $f_i$ are monic, pairwise coprime. Use Lemma \ref{0}ib and Lemma \ref{4} (i and ii) to see that the $n_i$ satisfy the requirements. 

iii $\implies$ i: Assume without loss of generality that this special $i$ is $m$ and
consider
$F=\prod_{i=1}^{m-1}R[x]/(f_i)$. We can now use Lemma \ref{4}iii, Lemma \ref{11} and Lemma \ref{12} inductively to see that $F$ is a futile
$R$-algebra.
Consider $F \times R[x]/(f_m)^{n_m}$. An $R$-subalgebra of $R[x]/(f_m)^{n_m}$ is isomorphic to $R$, $R[x]/(x^2)$ or $R[x]/(x^3)$ (Lemma
\ref{4}i). All these rings have finitely many quotients. We can again apply Lemma \ref{11} and Lemma \ref{12} to finish the proof.

iii $\implies$ ii: This is obvious if one uses the Chinese remainder theorem and if one takes $A'=R[x]/(f_i)^{n_i}$ for the special $i$ if it
occurs and $A'=0$ otherwise. 

ii $\implies$ iii: We may assume that $A'$ is local or $0$, since otherwise $A'=A'' \times R$ and we can put this $R$ in $\prod_i A_i$. We will first
see what such an $A'$ can be. Let $\mathrm{dim}_R(A')=r$. If $r=0$, then we obtain $A'=0$.
If $r=1$, then we find $A'=R$.
If $r=2, 3$, notice first that $\sqrt{0_A}=\mathfrak{n}$ is the maximal ideal of $A$. From our assumptions we get
$\mathrm{dim}_R(\mathfrak{n}/\mathfrak{n}^2)=1$. Using Nakayama's Lemma, we see that $\mathfrak{n}$ is principal, say $\mathfrak{n}=(a)$. Then
one has $A=R[a]$. By looking at dimension, we conclude that $A \cong R[x]/(x^r)$.   

Hence we see that $A \cong_R A'' \times \prod_{j=1}^m B_j$ where $A'' \cong_R R[x]/(x^i)$ where $i=2,3$ or $A''=0$ and the $B_j$ are primitive field
extensions of $R$. Let $f$ be an irreducible polynomial, then $R[x]/(f) \cong R[x]/(g)$ for infinitely many irreducible polynomials $g$. Indeed, for
$a \in R$ we have $R[x]/(f(x)) \cong R[x]/(f(x-a))$ and this gives us infinitely many different polynomials. Hence we can apply the Chinese remainder
theorem to see that iii holds.

\end{proof}

\subsection{Artinian rings} \label{801}

We will now consider the case where $R$ is an artinian ring. By Theorem \ref{120} we reduce directly to the case where $R$ is local.

\begin{lemma}[Nakayama] \label{96}
 Let $(R,\m)$ be a local artinian ring and let $M$ be an $R$-module. The following hold:
\begin{enumerate}
 \item Suppose that $M=\m M$. Then we have $M=0$.
 \item Suppose that $N \subseteq M$ is an $R$-submodule and suppose that $N+\m M=M$. Then we have $N=M$.
\end{enumerate}
\end{lemma}
\begin{proof}
 i. Note that $\m$ is nilpotent, say $\m^n=0$ (Proposition 8.4 from \cite{AT}). Then $M=\m M = \m^2 M = \ldots = \m^n M=0$. 
 
 ii. Apply i to $M'=M/N$. 
\end{proof}

Recall that an $R$-module $M$ is called uniserial if the set of $R$-submodules of $M$ is linearly ordered by inclusion.

\begin{lemma} \label{90}
 Let $(R,\m)$ be a local artinian ring and let $M$ be an $R$-module. Let $k=R/\m$. Then the following conditions are equivalent:
\begin{enumerate}
 \item $M$ is uniserial;
 \item $M$ is uniserial of finite length;  
 \item for all $n \in \Z_{\geq 0}$ we have $\mathrm{dim}_k(\m^n M /\m^{n+1}M) \leq 1$;
 \item for all $n \in \{0,1\}$ we have $\mathrm{dim}_k(\m^n M /\m^{n+1}M) \leq 1$.
\end{enumerate}
\end{lemma}
\begin{proof}
 i $\implies$ iii: Otherwise we have submodules between $\m^{n+1}M$ and $\m^n M$ without inclusions. 

 iii $\implies$ iv: Obvious.

 iii $\implies$ ii: Assume that $M \neq 0$. The case $n=0$ together with Lemma \ref{96} show that $M \cong_R R/I$ for some $R$-ideal $I$. The second
condition, by Lemma
\ref{96}, shows that $R/I$ is a principal ideal ring. Since an artinian ring has finite length, $M$ has finite length. One can easily prove that a
zero dimensional principal ideal ring has only finitely many ideals,
which are ordered linearly by inclusion, and hence that $M$ is uniserial.

ii $\implies$ i: Trivial. 
\end{proof}

\begin{remark} \label{133}
Let $(R,\m)$ be a local artinian ring and let $M$ be a uniserial $R$-module. Then the submodules of $M$ are just $M \supseteq \m M \supseteq \m^2 M
\supseteq \ldots$. 
\end{remark}

\begin{lemma} \label{93}
Let $f:B \to A$ be a morphism between artinian rings. For $\qa \in \mathrm{Spec}(B)$ we have 
\begin{eqnarray*}
A_{\qa} \cong \prod_{\pa \in \mathrm{Spec}(B): f^{-1}(\pa)=\qa} A_{\pa}.
\end{eqnarray*}
Furthermore, we have 
\begin{eqnarray*}
A= \prod_{\qa \in \mathrm{Spec}(B)} A_{\qa}. 
\end{eqnarray*}
\end{lemma}
\begin{proof}
 The first statement follows from the fact that $A_{\qa}$ is artinian (or zero), and hence a product of the localization at its prime ideals. The second
statement follows from $B=\prod_{\qa \in \mathrm{Spec}(B)} B_{\qa}$ and $A=A \otimes_B B$. 
\end{proof}

\begin{lemma} \label{94}
 Let $R$ be an artinian ring and let $A$ be a commutative $R$-algebra, finitely generated as $R$-module. Then we have the following bijection:
\begin{eqnarray*}
\varphi: \{R\textrm{-}\mathrm{subalgebras\ of\ }A\} &\to& \{(\sim,(B_{[\pa]})_{[\pa] \in \mathrm{spec}(R)/\sim}): \sim \textrm{equiv rel on }\mathrm{Spec}(A), \\
&& B_{[\pa]} \mathrm{\ a\ local\ } R\mathrm{-subalgebra\ of\ }\prod_{\qa \in [\pa]} A_{\qa} \}
\end{eqnarray*}
given by
\begin{eqnarray*}
B &\mapsto& ( \pa \sim \qa \mathrm{\ iff\ }\pa \cap B=\qa \cap B,  (B_{\pa \cap B})_{[\pa] \in \mathrm{Spec}(A)/\sim}).
\end{eqnarray*}
\end{lemma}
\begin{proof}
First note that any subalgebra of $A$ is artinian.
That the map makes sense, follows from Lemma \ref{93} and exactness of localization.
We will construct an inverse $\psi$ of the map above. It maps $(\sim,(B_{[\pa]})_{[\pa] \in \mathrm{spec}(R)/\sim})$ to $\prod_{[\pa] \in
\mathrm{Spec}(A)/\sim}  B_{[\pa]}$. As $B$ is artinian, one easily sees $\psi \circ \varphi(B)=B$ (Lemma \ref{93}). It also follows easily that the
other composition is the identity. 
\end{proof}

We have the following reduction theorem.

\begin{proposition} \label{105}
Let $(R,\m)$ be a local artinian ring such that $k=R/\m$ is infinite and let $A$ be an $R$-algebra. Then the following properties are
equivalent.
\begin{enumerate}
 \item $A$ is a futile $R$-algebra;
 \item $A$ is commutative, $A/\m A$ is a futile $k$-algebra, $\m (A/R)$ is a uniserial $R$-module and $R+\sqrt{0_A} \subseteq A$ is a futile $R$-algebra.
\end{enumerate}
\end{proposition}
\begin{proof}
 i $\implies$ ii: We need to show that the four properties hold. Number one follows from Theorem \ref{24}. Number two and four follow from Lemma
\ref{0}ia,ib. We will show that $\m( A/R)$ is a uniserial $R$-module by showing that iii from Lemma \ref{90} is satisfied. Let $n \in \Z_{\geq 1}$.
Note that the $R$-submodules of $\m^n
(A/R) /\m^{n+1} (A/R)$ correspond bijectively to $R$-submodules of $\left(\m^n A +R \right) / \left(\m^{n+1} A +R \right)$. Let $L$ be an $R$-module such that $\m^{n+1}A+R \subseteq L \subseteq \m^n A+R$. We claim that $L$ is an $R$-algebra. If $x+r, x'+r' \in L$, where
$x,x' \in \m^n A$, $r,r' \in R$, then $(x+r)(x'+r')=xx'+rx'+r'x+rr'$. Note that $xx' \in \m^{n+1}A$ as $n \geq 1$, $rx', r'x \in L$ as $L$ is an
$R$-module and that $rr' \in R$. Hence $L$ is indeed a ring. If $\mathrm{dim}_k( \left(\m^{n+1}A+R\right)/ \left(\m^n A +R \right)) >1$, then there
are infinitely many such $L$ since $k$ is infinite, which gives a contradiction with the assumption that $A$ is $R$-futile. 

 ii $\implies$ i: Let $B$ be an $R$-subalgebra of $A$. From Theorem \ref{3} we deduce that the futile $k$-algebra $A/\m A$ is finite as $k$-module.
From Nakayama's lemma (Lemma \ref{96}) it follows
that $A$ is a finite $R$-module. It follows that $B$ is artinian as well. 

Step i: We show that there are only finitely many local $R$-subalgebras of $A$. Let $(B,\mathfrak{n})$ be such a local $R$-subalgebra. 
Suppose that $B \supset \m A$. Then we have $B/\m A \subseteq A/\m A$ and there are only finitely many such $B$ by the assumption that $A/\m A$ is a
futile
$k$-algebra. 
Assume that $B \not \supset \m A$. Notice first that the map $\m A/\m \to (\m A + R )/R=\m(A/R)$ is an isomorphism (since $\m A \cap R = \m$,
look at nilpotents). Note that from $\m \neq \m A$ and Lemma \ref{96} one obtains $\m A \supsetneq \m^2 A+\m$. 
Hence by the uniseriality we have a chain $\m A \supsetneq \m^2 A +\m \supseteq B \cap \m A \supseteq \m$ (see Remark \ref{133}). From this we see that
$\m^2 A+ \m =\m^2 A+ \left(B \cap \m A\right)$ and the latter is obviously a $B$-module. Also $\m A$ is a $B$-module and it follows that $\m A/(\m^2 A +\m)
\cong_R k$ is a simple $R$-module and hence a simple $B$-module. Hence $B/\mathfrak{n} \cong k$ and it follows that $B \subseteq R + \sqrt{0_A}$. By
assumption we have only finitely many such $R$-algebras and this finishes the first step.

Step ii: From Lemma \ref{94} it is enough, as $\mathrm{Spec}(A)$ is finite, to show that there are only finitely many local subalgebras of $A'=\prod_{\pa
\in S} A_{\pa}$ for $S \subseteq \mathrm{Spec}(A)$. We show that $R \to A'$ still satisfies the conditions of ii, and then we are done by step i. Write $A=A' \times A''$. One has $A/\m
A = A'/\m A' \times A''/\m A''$ and hence $A'/\m A'$ is still a futile $k$-futile. We have a surjective map $\m (A/R) \to \m(A'/R)$ and hence 
$\m(A'/R)$ is still a uniserial $R$-module. Furthermore, we have a natural surjective morphism of $R$-algebras $R + \sqrt{0_A} \to R + \sqrt{0_{A'}}$
(obtained from the maps $R + \sqrt{0_A} \to A=A' \times A'' \to A'$). From Lemma \ref{0}ia it follows that $R+\sqrt{0_{A'}}$ is $R$-futile.
\end{proof}

Note that in the previous statement $R+\sqrt{0_A}$ is a local commutative $R$-algebra. The next proposition handles this case.

\begin{proposition} \label{122}
Let $(R,\m)$ be a local artinian ring such that $k=R/\m$ is infinite. Let $(A,\mathfrak{n})$ be a commutative local $R$-algebra with $A/\mathfrak{n}
=k$. Put $r_A= \mathrm{dim}_{k}(\sqrt{0_{A/\m A}})$. Then the following conditions are
equivalent.
\begin{enumerate}
 \item $A$ is a futile $R$-algebra;
 \item $A/\m A$ is a futile $k$-algebra, $\m (A/R)$ is a uniserial $R$-module, and if $r_{A}=2$, then
$\mathfrak{n}^4+\mathfrak{n}^2\m+\m=\m A$.
\end{enumerate}
\end{proposition}
\begin{proof}
ii $\limplies$ i: From Theorem \ref{3} it follows that $r_A \in
\{0,1,2\}$.

Let $B \subseteq A$ be an $R$-subalgebra. Let $\varphi_B: B \to A/\m A$ be the natural map. For all of the finitely many $k$-subalgebras $S$ of $A/\m A$ we show that there are only finitely many $B$ such that $\mathrm{Im}(\varphi_B)=S$. 

Suppose that $\mathrm{Im}(\varphi_B)=A/\m A$. It follows from Lemma \ref{96} that $A=B$.

Suppose that $\mathrm{Im}(\varphi_B)=k$. Then we have $B=R+(B \cap \m A)$. But $\m \subseteq B \cap \m A \subseteq \m A$, and as $\m A/\m=\m(A/R)$ is
a uniserial $R$-module, there are only finitely many options for $B \cap \m A$ and hence for $B$.

Suppose that $\mathrm{Im}(\varphi_B) \neq k, A/\m A$. Then we know from Theorem \ref{3} that $A/\m A \cong_k k[x]/(x^3)$, that
$r_A=2$, and that $\mathrm{Im}(\varphi_B)=k[x^2] \subset k[x]/(x^3)$ (Lemma \ref{4}). It follows that $B \subseteq
\varphi_B^{-1}(k[x^2])=R+\mathfrak{n}^2+\m A=:A'$,
the latter being
an local $R$-algebra with maximal ideal $\m_{A'}=\mathfrak{n}^2+\m A$. By construction we have $A'/\m A \cong_k k[x^2] \subset k[x]/(x^3)$, which is
of dimension $2$ over $k$. By the uniseriality
assumption we have $\mathrm{dim}_k(\m
A/(\m^2 A+\m))= \mathrm{dim}_k \left(\m (A/R)/\m^2(A/R)\right) \leq 1$. We have $\m A'=\m+\m \mathfrak{n}^2+\m^2 A$. Notice that $t=\mathrm{dim}_k(A'/\m A')=2+\mathrm{dim}_k(\m A/\left( \m + \m \mathfrak{n}^2+\m^2 A \right)) \leq 3$. 
Notice that $t=3$ iff $\m^2 A+\m \subsetneq \m A$ and $\m \mathfrak{n}^2 \subseteq \m^2A+\m$. 

Assume first that $t=2$. Then we have $A'/\m A' \cong_k k[x^2] \subset k[x]/(x^3)$. Notice furthermore that  $\m (A'/R) \subseteq \m (A/R)$ is uniserial and $A'/\m A'$ is a futile
$k$-algebra by Theorem \ref{3}. As $B \subseteq A'$, there are only finitely many options for $B$ by the cases where $\mathrm{Im}(\varphi_B)=k,A/\m
A$.

Assume that $t=3$. By Theorem \ref{3} the ring $A'/\m A'$ is $R$-futile if and only if the square of its maximal ideal is not zero. This is equivalent
to $\m_{A'}^2=\mathfrak{n}^4+\m^2 A + \mathfrak{n}^2 \m \not \subset \m A'=\m^2 A + \m$, and this holds by assumption. In this case, $A'/\m A'
\cong k[x]/(x^3)$. The map $B \to A'/\m A'$ is local, induces an isomorphism on the residue field, and $A'/\m A'$ is a finitely generated $B$-module.
Let $\m_B$ be the maximal ideal of $B$ (it is local by integrality). From $\varphi_B$ one gets $\m_B+\m A=\m_{A'}$ and from the map $A'/\m A' \to
A'/\m A$ one gets $\m A=\m A'+\m_{A'}^2$. Combining these gives that the map $\m_{B} \to \m_{A'}/\m_{A'}^2$ is surjective.
By Lemma 7.4 from \cite{HA}, the map $B \to A'/\m A'$ is surjective. From Nakayama's Lemma (Lemma \ref{96}) we conclude that $B=A'$.

i $\implies$ ii: The first three parts follow from Proposition \ref{105}. Assume $r_A=2$. Note that we have an inclusion
$\mathfrak{n}^4+\mathfrak{n}^2\m+\m^2A+\m\subseteq \m
A$ (Theorem \ref{3}, as $\mathfrak{n}^3 \subseteq \m A$). From the uniseriality it follows that either $\m^2 A+\m = \m A$, or that for every $x \in \m
A \setminus
\left(\m^2A +\m \right)$ we have $\m A=\m^2 A+ Rx$. So we are done unless $\mathfrak{n^4}+\mathfrak{n}^2 \m \subseteq \m^2 A+\m$ and $\m^2 A+\m
\subsetneq \m A$. Assume that we are in this case and consider
the ring $A'=R+\mathfrak{n}^2+\m A$ as above. Notice that $\mathrm{dim}_k(A'/\m A')=3$ (as $\m(A/R)$ is $R$-uniserial) and that $A'/\m A'$ is a local futile $k$-algebra (Lemma \ref{0}ia,ic). By Theorem \ref{3} we have $(\mathfrak{n}^2+\m A)^2
\not \subset  \m A'=\m^2 A+\m$ (Theorem \ref{3}), contradiction.
\end{proof}

 The condition $\mathfrak{n}^4+\mathfrak{n}^2\m+\m=\m A$ looks artificial, but one can give examples which show that all terms are needed.

\begin{proof}[Proof of Theorem \ref{913}]
Combine Proposition \ref{105} and Proposition \ref{122} and use that a submodule of a uniserial module is uniserial.
\end{proof}

\section{Principal ideal domains with finite quotients}

For certain rings $R$ one can find a nice description of the futile $R$-algebras. In this section we will handle the case where $R$ is a principal ideal domain with finite quotients. One can generalize this theory 
to for example discrete valuation rings, but since we have the general theory, there is no need for this. 

\begin{lemma} \label{16}
 A commutative ring $R$ that is a domain but not a field has infinitely many ideals. 
\end{lemma}
\begin{proof}
This follows from the fact that artinian domains are fields.
\end{proof}

\begin{lemma} \label{18}
 Let $R$ be a domain that is not a field. Let $A$ be an $R$-algebra. Assume that $A$ is a futile $R$-algebra, nonzero, and torsion-free as $R$-module.
Then we have
$R \subseteq A \subseteq Q(R)=K$.  
\end{lemma}
\begin{proof}
 Let $S=R \setminus\{0\}$. Then we have, as $A$ is torsion free, $A \subseteq S^{-1}A \cong K[x]/(f)$ for some nonzero $f \in K[x]$ (Lemma
\ref{0}iii and Theorem \ref{3},
where we note that finite domains are fields). After multiplying by elements of $S$ we may assume that $x \in A$ and $f \in R[x]$ monic. 
Division with remainder shows $K[x]f \cap R[x]=R[x]f$. This shows that we have $T=R[x]/(f) \subseteq A$. We will show that
$\mathrm{deg}(f)=1$. Consider the $R$-subalgebras $R+IT$ where $I$ is an ideal of $R$. If $\mathrm{deg}(f)>1$, one easily gets
$I=\mathrm{Ann}_R(T/(R+IT))$. This gives infinitely many $R$-subalgebras by Lemma \ref{16}, which contradicts the futility. Hence
$\mathrm{deg}(f)=1$ and $R \subseteq A \subseteq K$.
\end{proof}

Notice that the converse of the above lemma is false: the ring $\Q$ for example has infinitely many $\Z$-subalgebras. We have the following lemma.

\begin{corollary}
 Let $R$ be a domain that is not a field. Let $A$ be a futile $R$-algebra. Then $A \otimes_R Q(R)=0$ or $A \otimes_R Q(R)=Q(R)$. 
\end{corollary}
\begin{proof}
 Consider the exact sequence $0 \to A_{R\textrm{-tor}} \to A \to A/A_{R\textrm{-tor}} \to 0$ and tensor with $Q(R)$ over $R$. We get an isomorphism
$A
\otimes_R Q(R) \cong A/A_{R\textrm{-tor}} \otimes_R Q(R)$. Then apply Lemma \ref{18} and Lemma \ref{0}ib. 
\end{proof}

\begin{lemma} \label{19}
Let $R$ be a principal ideal domain. Then an $R$-subalgebra $A$ of $Q(R)$ is $R$-futile if and only if $A=R[\frac{1}{r}]$ for some
$r \in R \setminus \{0\}$. 
\end{lemma}
\begin{proof}
It is an easy exercise to show that there is a bijection between the set of $R$-subalgebras of $Q(R)$ and the powerset of $\mathrm{Spec}(R) \setminus \{0\}$ given by $A \mapsto  \{ \pa=(p): \frac{1}{p} \in A\}$. The result above then
follows easily. 
\end{proof}

\begin{theorem} \label{21}
Let $R$ be a principal ideal domain, not a field, such that the residue fields for all nonzero primes are finite. Then an $R$-algebra $A$ is a
futile $R$-algebra if one of the following holds:
\begin{itemize}
\item $A$ is finite;
\item $A_{R\textrm{-}\mathrm{tor}}$ is finite and $A/A_{R\textrm{-}\mathrm{tor}} \cong R[1/r] \subseteq Q(R)$ for some $r \in R \setminus \{0\}$.
\end{itemize}
\end{theorem}
\begin{proof}
$\implies$: Let $I=\mathrm{Ker}(R \to A)$. If $I \neq 0$, then $A$ is a futile $R/I$-algebra where $R/I$ is finite. By Theorem \ref{2} we conclude
that $A$ is
finite.
Suppose that $I=0$. We will first show that $A_{R\textrm{-}\mathrm{tor}}$ is finite. Indeed, for all $r \in R$ we have an
ideal $A[r]=\{a \in A: ra=0\}$. As $R \cap A[r]=0$, we see that $B_r:=R+A[r]=R \oplus A[r]$ is a subring of $A$. As $\left(B_r\right)_{R\textrm{-}\mathrm{tor}}=A[r]$, by futility there is $r \in R$ such that $A_{R\textrm{-}\mathrm{tor}}=A[r]$. For such an $r$ consider $B_r/rB_r= R/rR \oplus A_{R\textrm{-}\mathrm{tor}}$. This ring is
a futile $R/rR$-algebra (Lemma \ref{0}ic) and by Theorem \ref{2} it is finite. Hence $A_{R\textrm{-}\mathrm{tor}}$ is finite. We know that $A/A_{R\textrm{-}\mathrm{tor}}$ is torsion free
and is a futile $R$-algebra. By Lemma \ref{18} and Lemma \ref{19} we have $A/A_{R\textrm{-}\mathrm{tor}} \cong R[1/r]$ for some $r \in R$.

$\limplies$: If $A$ is finite, then it obviously is a futile $R$-algebra. For the other part, use Lemma \ref{17} and Lemma \ref{19}. 
\end{proof}

\begin{remark}
Let $p$ be a prime number. The above theorem holds for example for $R=\Z, \Z_{p}, \Z_{(p)}$.
\end{remark}

\end{document}